\documentclass[11pt]{scrartcl}
\usepackage[utf8]{inputenc}
\usepackage{amsmath,amssymb,amsthm}
\usepackage{enumerate}
\usepackage{url}
\newtheorem{thm}{Theorem}[section]
\newtheorem{dfn}[thm]{Definition}
\newtheorem{lem}[thm]{Lemma}
\newtheorem{prop}[thm]{Proposition}
\newtheorem{cor}[thm]{Corollary}

\theoremstyle{remark}
\newtheorem{rem}[thm]{Remark}
\newtheorem{ex}[thm]{Example}
\newtheorem{claim}[thm]{Claim}

\DeclareMathOperator{\scal}{Scal}
\DeclareMathOperator{\ric}{Ric}
\DeclareMathOperator{\R}{R}
\DeclareMathOperator{\Lop}{L}

\DeclareMathOperator{\Wop}{W}

\DeclareMathOperator{\tr}{trace}
\DeclareMathOperator{\Id}{I}
\DeclareMathOperator{\id}{id}

\newcommand{\ddt}[1]{\frac{\partial #1 }{\partial t}}
\newcommand{\ps}[2]{\left\langle #1,#2 \right\rangle}
\newcommand{\eps}{\varepsilon}

\title{Noncoercive Ricci flow invariant curvature cones.}
\author{Thomas Richard and Harish Seshadri}
\begin{document}
\maketitle
\begin{abstract}
  This note is a study of nonnegativity conditions on curvature which are
  preserved by the Ricci flow. We focus on specific kinds of
  curvature conditions which we call noncoercive, these are the
  conditions for which nonnegative curvature and vanishing scalar
  curvature doesn't imply flatness.

  We show that, in dimensions greater
  than $4$, if a Ricci flow invariant
  condition is weaker than ``Einstein with nonnegative scalar
  curvature'', then this condition has to be (if not void) the
  condition ``nonnegative scalar curvature''. As a corollary, we
  obtain that a Ricci flow invariant curvature condition which is
  stronger than ``nonnegative scalar curvature'' cannot be (strictly)
  satisfied by compact Einstein symmetric spaces such as
  $\mathbb{S}^2\times\mathbb{S}^2$ or $\mathbb{CP}^2$.

  We also investigate conditions which are satisfied by all
  conformally flat manifolds with nonnegative scalar curvature.
\end{abstract}
When studying Ricci flow, it is useful to know that some
``nonnegative curvature''-type geometric condition is preserved along
the flow. For instance the proof by Brendle and Schoen of the
differentiable sphere theorem (\cite{MR2449060}) has been made possible by the proof that
the PIC condition is preserved (indepently proved in \cite{MR2449060} and \cite{MR2587576}).

Although Ricci flow has been studied extensively since R. Hamilton's seminal paper, there is still no
comprehensive theory of curvature conditions which are preserved by Ricci flow. A significant advance in this direction is  the work of Wilking (\cite{2010arXiv1011.3561W}) which gives
a unified construction for almost all known Ricci flow invariant
curvature conditions. The paper \cite{2011arXiv1101.5884G}
gives general results on curvature conditions coming from this construction.

We want to gain a better understanding
of general Ricci flow invariant curvature conditions. Curvature conditions are
encoded by convex cones $\mathcal{C}$ (called curvature cones) in the space of
algebraic curvature operators $S^2_B\Lambda^2\mathbb{R}^n$ which are invariant under the natural
action of the orthogonal group. As a consequence of the maximum principle for systems, a sufficient condition for a
curvature condition to be preserved under the Ricci flow is  the preservation of
the cone $\mathcal{C}$ by the flow of some explicit vector field.
Readers not familiar with these notions will find a quick exposition
and references in section \ref{sec:algebr-curv-oper}.

The largest Ricci flow invariant curvature cone is the cone
$\mathcal{C}_{\scal}$ of curvature operators with nonnegative trace, which
geometrically translates to the condition ``nonnegative scalar
curvature''. Other examples of Ricci flow invariant curvature cones
include the cone of operators which are nonnegative (as symmetric
operators on $\Lambda^2\mathbb{R}^n$), the cone of $2$-nonnegative
curvature operators, the cone of curvature operators with positive
isotropic curvature and the cones ``PIC1'' and ``PIC2'' which are
built from it. An overview of the relations between these condtions
and their geometric implications can be found in \cite{MR2583938}

In this paper, we are interested with curvature cones which are
``non-coercive". We say that a curvature cone $\mathcal{C}$ is {\it non-coercive} if it
contains a nontrivial vector subspace. This condition is equivalent to the existence of a
non-vanishing curvature operator $\R$ in $\mathcal{C}$ whose scalar
curvature is zero (See Section 1 for other characterizations of non-coercive cones). Non-coercive Ricci flow invariant cones seem quite rare. In dimension 5 and above, the only known example is the cone $\mathcal{C}_{\scal}$.

For representation theoretic reasons, non-coercive cones fall into four classes :
\begin{enumerate}
\item $\mathcal{C}$ is the full space $S^2_B\Lambda^2\mathbb{R}^n$ of algebraic curvature operators.
\item $\mathcal{C}$ is the cone $\mathcal{C}_{\scal}$ of curvature operators with nonnegative scalar curvature.
\item $\mathcal{C}$ contains all Ricci flat (also known as ``pure Weyl") curvature operators.
\item $\mathcal{C}$ contains all pure traceless Ricci (also known as ``scalar flat and conformally flat") curvature operators.
\end{enumerate}
This is explained in Section 1. 

Our first result shows that there are no Ricci flow invariant cones in the third class, except the cone $\mathcal{C}_{\scal}$. In fact, by using a simple fact from representation theory (Proposition A.5), we have the following stronger statement that the existence of a single nonzero Ricci flat curvature operator is enough to obtain
this conclusion:

\begin{thm}\label{thm:main}
Let $\mathcal{C}\subset S^2_B\Lambda^2\mathbb{R}^n$, $n\geq 4$, be a
  closed, convex, $O(n)$-invariant, Ricci flow invariant curvature cone which
  contains the identity operator in its interior. If $\mathcal{C}\setminus \{0\}$ contains a
Ricci flat metric then $\mathcal{C}$ is either the whole space
  $S^2_B\Lambda^2\mathbb{R}^n$ or the cone $\mathcal{C}_{\scal}$
\end{thm}

\begin{rem}
The above theorem allows us to weaken the hypotheses of Theorem 3  in S. Brendle's paper \cite{BrendleEinsteinRFinvCurv} the statement of which is as follows: Let
$\mathcal{C} \subset S^2_B\Lambda^2\mathbb{R}^n$ be a closed, convex, $O(n)$-invariant cone  which is preserved by the Ricci flow. Further assume that the identity operator lies in the interior of $\mathcal{C}$ and that every element of $\mathcal{C} \setminus \{0\}$ has nonnegative scalar curvature and nonzero Ricci tensor. If $(M,g)$ is a compact Einstein $n$-manifold  whose curvature operator at any point lies in the interior of $\mathcal{C}$, then $(M,g)$ has constant sectional curvature.

By Theorem \ref{thm:main} we can drop the assumption that every element of $\mathcal{C}$ has nonzero Ricci tensor and just demand that $\mathcal{C} \neq \mathcal{C}_{\scal}$.

\end{rem}

As another corollary, we obtain a second result which explains the following observation that one can make
about Ricci flow invariant curvature cones:
For most Ricci flow invariant curvature conditions, $\mathbb{CP}^{n/2}$ is
not positively curved, only nonnegatively curved. For instance,
$\mathbb{CP}^{n/2}$ has nonnegative curvature operator and isotropic
curvature but doesn't have positive curvature operator or isotropic
curvature. The corollary below shows that ``nonnegative scalar
curvature'' is the only Ricci flow invariant curvature condition for which
$\mathbb{CP}^{n/2}$ is positively curved.

A symmetric space is said to be nontrivial if it doesn't have constant
sectional curvature.

\begin{cor}\label{cor_coneWithEinsteinSymInInterior}
  Let $\mathcal{C}\subset S^2_B\Lambda^2\mathbb{R}^n$, $n\geq 4$, be a Ricci flow invariant curvature cone
  which contains the curvature operator of a nontrivial Einstein
  symmetric space of non-negative scalar curvature, then $\mathcal{C}$
  is either the whole space $S^2_B\Lambda^2\mathbb{R}^n$ or the cone
  $\mathcal{C}_{\scal}$.
\end{cor}
\begin{rem}
As a byproduct, we obtain that in any even dimension $n$, the cone $\mathcal{C}_{\scal}$ is the only cone which contains all curvature operators with nonnegative sectional curvature, because the curvature operator of $\mathbb{CP}^{n/2}$ must be in the interior of such a cone since it has positive sectional curvature.
\end{rem}

Next, we study Ricci flow invariant cones which fall into the fourth class, that is cones which contain all conformally flat scalar flat curvature operators. It turns out that the situation there is more complicated. At least in dimension 4, there are Ricci flow invariant cones which contain all conformally flat scalar flat curvature operators, the first example is the cone of curvature operators whose isotropic curvature is nonnegative, (see \cite{MR924677}). However, the PIC cone doesn't contain all conformally flat scalar flat curvature operators in dimension 5 and above. Our first result about cones in the fourth class is that they are quite common in dimension 4:
\begin{thm}\label{thm_ex_dim4}
Let $\tilde{\mathcal{C}}\subset S^2_B\Lambda^2\mathbb{R}^4$ be any Ricci flow invariant cone which
contains all nonnegative curvature operators and  let :
\[\mathcal{C}=\left\{\R\left|\frac{\scal}{12}+W\in\tilde{\mathcal{C}}\right .\right\}\subset S^2_B\Lambda^2\mathbb{R}^4\]
where $\scal$ and $W$ are the scalar curvature and the Weyl curvature components of $\R$.

  Then $\mathcal{C}$ is a Ricci flow invariant cone which contains all conformally flat scalar flat curvature operators.
\end{thm}
It turns out that some of these cones can be recovered by Wilking's
construction (\cite{2010arXiv1011.3561W}) while some others are
genuinely new, see Remark \ref{sec:rem_wilk_cones}.

As mentioned earlier, in dimension 5 and above, there is not a single known Ricci flow invariant cone which contain all conformally flat scalar flat curvature operator. We prove the following restriction on such a cone:
\begin{thm}\label{thm_cone_with_cfsf}
Let $\mathcal{C}$ be a Ricci flow invariant cone which contains a nonzero conformally flat scalar flat curvature operator and all nonnegative curvature operators, then $\mathcal{C}=\mathcal{C}_{\scal}$.
\end{thm}

The paper is organised as follows. First, we give a few
definitions that give us an abstract framework to talk about
curvature conditions, and which enable us to precisely define what a
``Ricci flow invariant curvature condition'' is.  In a second section, we prove some elementary propositions about non-coercive curvatures cones. The third section is devoted to the
proof of Theorem \ref{thm:main} and the fourth to the proof of
Corollary \ref{cor_coneWithEinsteinSymInInterior}. We treat the case of cones which contain all conformally flat scalar flat curvature operators in the fifth section. In an appendix,
we gather some elementary facts about convex cones which are invariant
under the action of a Lie group.

\section{Algebraic curvature operators, curvature cones and the Ricci flow }
\label{sec:algebr-curv-oper}

\begin{dfn}
  The space of algebraic curvature operators
  $S^2_B\Lambda^2\mathbb{R}^n$ is the space of symmetric endomorphisms
  $\R$ of $\Lambda^2\mathbb{R}^n$ which satisfy the first Bianchi identity:
  \[\forall x,y,z,t\in\mathbb{R}^n\quad \ps{\R(x\wedge y)}{z\wedge t}
  +\ps{\R(z\wedge x)}{y\wedge t}+\ps{\R(y\wedge z)}{x\wedge t}=0.\]
\end{dfn}
\begin{rem}
  Here, as in the rest of the paper, $\mathbb{R}^n$ is endowed with
  its standard euclidean structure, and the scalar product on $\Lambda^2\mathbb{R}^n$ is
  the one which comes from the standard one on $\mathbb{R}^n$ by the
  following construction:
  \[\ps{x\wedge y}{z\wedge t}=\ps{x}{z}\ps{y}{t}-\ps{x}{t}\ps{y}{z}.\]
  The same remark will hold when we will be considering spaces like
  $\Lambda^2TM$ where $(M,g)$ is a Riemannian manifold. $\Lambda^2TM$
  will be equipped with the euclidean structure
  coming from the euclidean structure on $TM$ given by the Riemannian metric.
\end{rem}

The space of algebraic curvature operators is the space of (pointwise)
tensors which satisfy the same symmetries as the Riemann curvature
tensor of Riemannian manifold. As in the case of Riemannian manifold,
it is interesting to consider the Ricci morphism:
$\rho:S^2_B\Lambda^2\mathbb{R}^n\to S^2\mathbb{R}^n$ which associates
to an algebraic curvature operator $\R$ its Ricci tensor which is a
symmetric operator on $\mathbb{R}^n$ defined by:
\[\ps{\rho(\R)x}{y}=\sum_{i=1}^n \ps{\R(x\wedge e_i)}{y\wedge e_i}\]
where $(e_i)_{1\leq i\leq n}$ is an orthonormal basis of
$\mathbb{R}^n$. $\R$ is said to be Einstein if $\rho(\R)$ is a
multiple of the identity operator $\id:\mathbb{R}^n\to\mathbb{R}^n$. Similarly, the scalar curvature of an algebraic
curvature operator is just twice its trace.

The action of $O(n,\mathbb{R})$ on $\mathbb{R}^n$ induces the following action of $O(n,\mathbb{R})$ on $S^2_B\Lambda^2\mathbb{R}^n$:
\begin{equation}
\ps{g.\R(x\wedge y)}{z\wedge t}=\ps{\R(gx\wedge gy)}{gz\wedge gt}.\label{eq:action}
\end{equation}

Recall that the representation of $O(n,\mathbb{R})$ given by its
action on $S^2_B\Lambda^2\mathbb{R}^n$ is decomposed into irreducible
representations in the following way:
\begin{equation}
S^2_B\Lambda^2\mathbb{R}^n=\mathbb{R}\Id\oplus
(S^2_0\mathbb{R}^n\wedge\id)\oplus \mathcal{W}\label{eq:decomp}
\end{equation}
where the space of Weyl curvature operators $\mathcal{W}$ is the
kernel of the Ricci endomorphism $\rho:S^2_B\Lambda^2\mathbb{R}^n\to
S^2\mathbb{R}^n$ and $S^2_0\mathbb{R}^n\wedge\id$ is the image of the space of traceless
endomorphims of $\mathbb{R}^n$ under the application $A_0\mapsto
A_0\wedge\id$. The wedge product of two symmetric operators
$A,B:\mathbb{R}^n\to \mathbb{R}^n$ is defined by :
\[(A\wedge B)(x\wedge y)=\frac{1}{2}\left ( Ax\wedge By+Bx\wedge
  Ay\right).\]
This corresponds to the half of the Kulkarni-Nomizu product of $A$ and
$B$ viewed as quadratic forms. 
In dimension 2, only the first summand of
\eqref{eq:decomp} exists. In dimension $3$ the $\mathcal{W}$
factor is $0$. Starting in dimension 4, all three components exist.

When needed, we will write $\R=\R_{\Id}+\R_0+\R_\mathcal{W}$ the
decomposition of a curvature operator along the three irreducible
components of \eqref{eq:decomp}.

\begin{dfn}
  A (nonnegative) curvature cone is a closed convex cone $\mathcal{C}\subset S^2_B\Lambda^2\mathbb{R}^n$ such that:
  \begin{itemize}
  \item $\mathcal{C}$ is invariant under the action of
    $O(n,\mathbb{R})$ given by \eqref{eq:action}.
  \item The identity operator
    $\Id:\Lambda^2\mathbb{R}^n\to\Lambda^2\mathbb{R}^n$ is in the
    interior of $\mathcal{C}$.
  \end{itemize}
\end{dfn}
\begin{rem}
  The condition that $\Id$ is in the interior of $\mathcal{C}$ implies
  that $\mathcal{C}$ has full dimension.
\end{rem}

This definition can be tracked back to the article \cite{MR1297501} of
M. Gromov. One should notice that we require the cone to be invariant
under the full orthogonal group $O(n,\mathbb{R})$, rather than under
the special orthogonal group $SO(n,\mathbb{R})$. For the result we
prove in this paper, this makes a difference only in dimension 4, where the action of $SO(4,\mathbb{R})$ on the space of Weyl tensors is not irreducible. The behavior of these ``oriented" curvature cones will be addressed in another paper.

Each of these cones defines a nonnegativity condition for the curvature
of Riemannian manifold in the following way: the curvature operator
$\R$ of a Riemannian manifold $(M,g)$ is a section of the bundle $S^2_B\Lambda^2TM$
which is built from $TM$  the same way $S^2_B\Lambda^2\mathbb{R}^n$ is
built from $\mathbb{R}^n$. For each $x\in M$, one can choose a
orthonormal basis of $T_xM$ to build an isomorphism between
$S^2_B\Lambda^2T_xM$ and $S^2_B\Lambda^2\mathbb{R}^n$. Thanks to the
$O(n,\mathbb{R})$-invariance of $\mathcal{C}$, this allows us to embed
$\mathcal{C}$ in $S^2_B\Lambda^2T_xM$ in a way which is independent
of the basis of $T_xM$ we started with.

We then say that $(M,g)$ has
$\mathcal{C}$-nonnegative curvature if for any $x\in M$ the curvature
operator of $(M,g)$ at $x$ belongs to the previously discussed
embedding of $\mathcal{C}$ in $S^2_B\Lambda^2T_xM$. Similarly, $(M,g)$
is said to have positive $\mathcal{C}$-curvature if its curvature
operator at each point is in the interior of $\mathcal{C}$. By
definition, the sphere $\mathbb{S}^n$ has positive
$\mathcal{C}$-curvature for all curvature cones $\mathcal{C}$.

This setting captures all the usual nonnegativity conditions which are
studied in Riemannian geometry, such as nonnegative scalar curvature,
nonnegative Ricci curvature, nonnegative sectional curvature and
nonnegative curvature operator. For instance, the cone which gives rise to the
nonnegative scalar curvature condition is just the half space of
$S^2_B\Lambda^2\mathbb{R}^n$ given by $\{\R\in S^2_B\Lambda^2\mathbb{R}^n|\tr(\R)\geq0\}$.

We now consider the interplay between these curvature cones and the
Ricci flow. If $(M,g(t))$ is a Ricci flow, Hamilton has proved in \cite{MR862046} that
the curvature operator $\R_{g(t)}$ of $(M,g(t))$ satisfies the
following evolution equation:
\[\ddt{\R_{g(t)}}=\Delta_{g(t)}\R_{g(t)}+2Q(\R_{g(t)})\]
where $Q$ is the $O(n,\mathbb{R})$ quadratic vector field on
$S^2_B\Lambda^2\mathbb{R}^n$ defined by:
\[Q(\R)=\R^2+\R^\#.\]
Here, $\R^2$ is just the square of $\R$ seen as an endomorphism of
$\Lambda^2\mathbb{R}^n$.  $\R^\#$ is defined in the following way:
\[\ps{\R^\#\eta}{\eta}=-\frac{1}{2}\sum_{i=1}^{n(n-1)/2}\ps{\left[\eta,\R
    \Bigl (\left[\eta,\R \left
          (\omega_i\right)\right]\Bigr)\right]}{\omega_i}\]
where $(\omega_i)_{i=1\dots n(n-1)/2}$ is an orthonormal basis of
$\Lambda^2\mathbb{R}^n$ and the Lie bracket $[\ ,\ ]$ on $\Lambda^2\mathbb{R}^n$ comes from its
identification with $\mathfrak{so}(n,\mathbb{R})$ given by:
\[x\wedge y\mapsto (u\mapsto \ps{x}{u}y-\ps{y}{u}x).\]
This expression for $\R^\#$ can be found in \cite{MR2415394}.

We will sometimes use the bilinear map $B$ associated to the quadratic map
$Q$, it is defined in the usual way :
\[B(\R_1,\R_2)=\frac{1}{2}\bigl(Q(\R_1+\R_2)-Q(\R_1)-Q(\R_2)\bigr). \]

We are now ready to define what a Ricci flow invariant curvature cone
is.

\begin{dfn}
  A curvature cone $\mathcal{C}$ is said to be Ricci flow invariant if for any
  $\R$ in the boundary $\partial\mathcal{C}$ of $\mathcal{C}$,
  $Q(\R)\in T_{\R}\mathcal{C}$, the tangent cone at $\R$ to $\mathcal{C}$.
\end{dfn}
\begin{rem}
  This condition is equivalent to the fact that the solutions to the ODE
  $\frac{d}{dt}\R=Q(\R)$ which start inside $\mathcal{C}$ stay in
  $\mathcal{C}$ for positive times.
\end{rem}

Hamilton's maximum principle (see \cite{MR862046}) implies :
\begin{thm}
  Let $\mathcal{C}$ be a Ricci flow invariant curvature cone.
  If $(M,g(t))_{t\in[0,T)}$ is a Ricci flow on a compact manifold such
  that $(M,g(0))$ has $\mathcal{C}$-nonnegative curvature, then for
  $t\in[0,T)$, $(M,g(t))$ has $\mathcal{C}$-nonnegative curvature.
\end{thm}
\begin{rem}
  It could happen that a nonnegativity condition is preserved under
  the Ricci flow while the associated cone is not Ricci flow invariant
  according to our definition, however such examples are not known to
  exist, as far as the knowledge of the authors go.
\end{rem}

\section{Elementary properties of non-coercive curvature cones}
\label{sec:elem-non-coer}
We prove here some properties of non-coercive curvature cones.

\begin{dfn}
A curvature cone is said to be non-coercive if it contains a nontrivial vector subspace.
\end{dfn}
\begin{ex}
  $\mathcal{C}_{\scal}$, the cone of curvature operators with
  nonnegative scalar curvature is an example of a non-coercive curvature
  cone.
\end{ex}
The $O(n,\mathbb{R})$ invariance of curvature cones gives the following :
\begin{prop}
  Let $\mathcal{C}$ be a non-coercive curvature cone, and
  $\mathcal{V}$ be the biggest vector space included in $\mathcal{C}$,
  then one of the following holds :
  \begin{enumerate}
  \item $\mathcal{V}=S^2_B\Lambda^2\mathbb{R}^n$,
  \item $\mathcal{V}=S^2_0\mathbb{R}^n\wedge\id\oplus\mathcal{W}$, in
    this case $\mathcal{C}=\mathcal{C}_{\scal}$,
  \item $\mathcal{V}=S^2_0\mathbb{R}^n\wedge\id$,
  \item $\mathcal{V}=\mathcal{W}$.
  \end{enumerate}
\end{prop}
\begin{proof}
  Thanks to Proposition \ref{lem_maxVectorSpace}, we know that
  $\mathcal{V}$ exists and is an $O(n,\mathbb{R})$ invariant subspace of
  $S^2_B\Lambda^2\mathbb{R}^n$. Thus it is a direct sum of some of the
  elements of decomposition \eqref{eq:decomp} (which are irreducible
  and pairwise non isomorphic). The only thing to check
  is that if $\mathbb{R}\Id\subset\mathcal{V}$, then
  $\mathcal{V}=S^2_B\Lambda^2\mathbb{R}^n$. This follows from the fact
  that $\Id$ is in the interior of $\mathcal{C}$, thus if $-\Id$ is in
  $\mathcal{C}$ then $0$ is in the interior of $\mathcal{C}$ and
  $\mathcal{C}=S^2_B\Lambda^2\mathbb{R}^n$.

  The fact that $\mathcal{C}=\mathcal{C}_{\scal}$ in the second case
  comes from the fact
  $\mathcal{V}=S^2_0\mathbb{R}^n\wedge\id\oplus\mathcal{W}$ is an
  hyperplane included in the boundary of $\mathcal{C}$ (Proposition
  \ref{prop_VboundaryC}). Thus $\mathcal{C}$ has to be one side of the
  hyperplane, since $\mathcal{C}$ contains the identity, $\mathcal{C}=\mathcal{C}_{\scal}$.
\end{proof}
\begin{prop}
Let $\mathcal{C}$ be a curvature cone. Then the following statements are equivalent :
\begin{enumerate}
\item $\mathcal{C}$ is non-coercive.
\item $\mathcal{C}$ contains a non zero curvature operator whose scalar curvature is zero.
\item $\mathcal{C}\cap \{\R|\tr(\R)\leq 1\}$ is not bounded.
\end{enumerate}
\end{prop}
\begin{rem}
  The third characterisation is useful in applications. It allows, when
  one knows that a manifold $(M,g)$ has nonnegative
  $\mathcal{C}$-curvature, and $\mathcal{C}$ is coercive, to get a bound on the full curvature tensor
  from an upper bound on the scalar curvature.
\end{rem}
\begin{proof}
  (1) implies (2) comes from the previous proposition, (2) implies (3)
  is easy by scaling the nonzero curvature operator in $\mathcal{C}$
  with zero scalar curvature.

  We show that (3) implies (2). The hypothesis tells us there is
  a sequence $\R_i$ of elements of $\mathcal{C}$ whose norm tends to
  infinity. Set $\tilde{\R}_i=\frac{\R_i}{\|\R_i\|}$. This is a
  bounded sequence of elements of $\mathcal{C}$. Up to a subsequence,
  it converges to a curvature operator which is in $\mathcal{C}$, has
  norm 1 and zero scalar curvature.

  We now show that (2) implies (1). Let $\mathcal{C}'=\mathcal{C}\cap
  (S^2_0\mathbb{R}^n\wedge\id\oplus\mathcal{W})$. This is an
  $O(n,\mathbb{R})$ cone of
  $S^2_0\mathbb{R}^n\wedge\id\oplus\mathcal{W}$. By Proposition
  \ref{prop_coneInIrredRep}, we have that $\mathcal{C}'$ is a vector
  space, thus $\mathcal{C}$ is non-coercive.
\end{proof}

\section{Curvature cones containing a Ricci flat operator}
\label{sec:proof-main-result}
This section gives the proof of Theorem \ref{thm:main}.

We will need the following lemma which shows how the quadratic vector field
$Q$ and its associated bilinear map $B$ act on the different parts of the decomposition of $S^2_B\Lambda^2\mathbb{R}^n$ in \eqref{eq:decomp}. This lemma is due to Böhm and Wilking (\cite{MR2415394}).
\begin{lem}\label{lem_QandIrreds}
\begin{itemize}
  \item $Q(\Id)=(n-1)\Id$.
  \item If $\Wop\in\mathcal{W}$, $B(\Wop,\Id)=0$.
  \item If $\Wop\in\mathcal{W}$ and $\R_0\in S^2_0\mathbb{R}^n\wedge\id$, then $B(\R_0,\Wop)\in S^2_0\mathbb{R}^n\wedge\id$.
  \item If $\R_0\in S^2_0\mathbb{R}^n\wedge\id$, then $B(\R_0,\Id)\in
    S^2_0\mathbb{R}^n\wedge\id$.
  \end{itemize}
\end{lem}

\begin{lem}\label{lem_computeBR0W}
  Let $\R_0$ be the traceless Ricci part of the curvature operator of
  $\mathbb{S}^{n-2}\times\mathbb{H}^2$ with its product metric where the
  first factor has constant curvature $1$ and the second has constant
  curvature $-1$.
  Define $\Wop=Q(\R_0)_{\mathcal{W}}$, the Weyl part of $Q(\R_0)$. Then there
  exists $a> 0$ such that :
  \[B(\R_0,\Wop)=a\R_0.\]
\end{lem}
\begin{rem}
  It is of course possible to directly compute $B(\R_0,\Wop)$ to prove
  the result and get the exact value of $a$. However, the calculation
  involves various constants depending on $n$ whose expression is a
  bit involved. The proof we provide bypasses this difficulty, at the
  cost of not providing an explicit value for $a$.
\end{rem}
\begin{proof}
  Let $\Lop=B(\R_0,\Wop)$. We first prove :
  \begin{claim}\label{claim_Ldif0}
    $\ps{\Lop}{\R_0}>0$.
  \end{claim}
  Recall that the trilinear map :
  \[(\R_1,\R_2,\R_3)\mapsto \ps{B(\R_1,\R_2)}{\R_3}\] is symmetric in
  all its three entries (see \cite{MR2415394}). Thus :
  $\ps{\Lop}{\R_0}=\ps{B(\R_0,\Wop)}{\R_0}=\ps{B(\R_0,\R_0)}{\Wop}$,
  and we have that :
  \[\ps{\Lop}{\R_0}=\ps{B(\R_0,\Wop)}{\R_0}=\|Q(\R_0)_{\mathcal{W}}\|^2.\]
  So we just need to show that $\Wop=Q(\R_0)_{\mathcal{W}}$ is not zero.

  Let us denote by $\R=\R_{\Id}+\R_0$ the curvature operator of
  $\mathbb{S}^{n-2}\times\mathbb{H}^2$. Then :
  \[Q(\R)=Q(\R_{\Id})+2B(\R_{\Id},\R_0)+Q(\R_0).\]
  By Lemma \ref{lem_QandIrreds}, the first two terms on the right hand
  side of the equality have no Weyl part. This implies that
  $Q(\R_0)_{\mathcal{W}}=Q(\R)_{\mathcal{W}}$.

  Since $\R$ is the curvature operator of
  $\mathbb{S}^{n-2}\times\mathbb{H}^2$ and $Q$ respects product
  structures, $Q(\R)$ is the curvature  operator of a product
  metric on $\mathbb{S}^{n-2}\times\mathbb{S}^2$ where the first
  factor has constant curvature $n-3$ and the second factor has
  constant curvature $1$. Note that this metric is not conformally flat,
  in particular $0\neq
  Q(\R)_{\mathcal{W}}=Q(\R_0)_{\mathcal{W}}$. Claim
  \ref{claim_Ldif0} is proved.

  It remains to prove that $\Lop$ is colinear to $\R_0$.

  We will use the following fact :

  \begin{claim}\label{claim_carac_R0}
    Write $\mathbb{R}^n$ as the direct sum $E\oplus F$ with $E=\{x\ |\
    x_{n-1}=x_n=0\}$ and
    $F=\{x\ |\ x_1=\dots=x_{n-2}=0\}$, and assume that a curvature
    operator $\tilde{\R}\in S^2_0\mathbb{R}^n\wedge\id$ admits
    $E\wedge E$, $E\wedge F$ and $F\wedge F$ as eigenspaces. Then
    $\tilde{\R}$ is a multiple of $\R_0$.
  \end{claim}
  To prove this, we write $\tilde{\R}$ as
  $\frac{2}{n-2}\tilde{\ric}_0\wedge\id$, where $\tilde{\ric}_0$ is
  the Ricci tensor of $\tilde{\R}$. It is straightforward to see that
  $E\wedge E$, $E\wedge F$ and $F\wedge F$ are eigenspaces of
  $\tilde{\R}$ if and only if $E$ and $F$ are eigenspaces of
  $\tilde{\ric}_0$.

  Moreover, the space of traceless Ricci tensors
  which have $E$ and $F$ as eigenspaces is of dimension 1 (once the
  eigenvalue of $E$ is chosen, the tracelessness imposes the
  eigenvalue on $F$). This
  shows that the conditions we have imposed on $\tilde{\R}$ describe a
  vectorial line in the space of curvature operators. Since $\R_0$
  also satisfies these conditions, Claim \ref{claim_carac_R0} is proved.

  Using the previous claim, we just need to show that $\Lop$ is in
  $S^2_0\mathbb{R}^n\wedge\id$ and admits $E\wedge E$, $E\wedge F$ and
  $F\wedge F$ as eigenspaces. Writing $\Lop=B(\R_0,\Wop)$, Lemma
  \ref{lem_QandIrreds} ensures that $\Lop\in
  S^2_0\mathbb{R}^n\wedge\id$.

  To see that the second hypothesis is
  fulfilled, we make the following observation : if an algebraic
  curvature operator $\R$ admits $E\wedge E$, $E\wedge F$ and
  $F\wedge F$ as eigenspaces, so do $\R_{\Id}$, $\R_0$,
  $\R_\mathcal{W}$, and $Q(\R)$. This is obvious for $\R_{\Id}$. For
  $\R_0$, just notice that the Ricci tensor of $\R$ has $E$ and $F$ as
  eigenspaces. Writing $\R_\mathcal{W}=\R-\R_{\Id}-\R_0$ proves the
  assertion for $\R_\mathcal{W}$. For $Q(\R)$, this is just a
  computation using the definition of $Q$.

  We can now prove that $\Lop$ admits $E\wedge E$, $E\wedge F$ and
  $F\wedge F$ as eigenspaces. First notice that by the previous
  observation, $\Wop=Q(\R_0)_\mathcal{W}$ has eigenspaces $E\wedge E$,
  $E\wedge F$ and $F\wedge F$. Then write :
  \[\Lop=B(\R_0,\Wop)=\frac{1}{2}\bigl (
  Q(\R_0+\Wop)-Q(\R_0)-Q(\Wop)\bigr),\]
  and notice that all the terms on the right hand side have eigenspaces $E\wedge E$,
  $E\wedge F$ and $F\wedge F$. This shows that $\Lop$ satisfies the
  assumptions of Claim \ref{claim_carac_R0} and concludes the proof.
\end{proof}

By Proposition \ref{prop_coneInIrredRep}, the existence of a nonzero Ricci flat operator in $\mathcal{C}$ implies that
$\mathcal{W} \subset \mathcal{C}$. Hence Theorem
\ref{thm:main} is a consequence of the following proposition :
\begin{prop}
 If $\mathcal{C}$ is a Ricci flow invariant curvature cone
  containing $\mathcal{W}$, then $\mathcal{C}$ is
  either the cone of curvature operators with nonnegative scalar
  curvature or the whole space $S^2_B\Lambda^2\mathbb{R}^n$.
\end{prop}
\begin{proof}
  We assume that $\mathcal{C}$ is not
  $S^2_B\Lambda^2\mathbb{R}^n$. Since
  $\mathbb{R}_+\Id+\mathcal{W}\subset\mathcal{C}$, we have that the
  vector space $\mathcal{V}$ which is defined in
  Proposition~\ref{lem_maxVectorSpace} satisfies
  $\mathcal{W}\subset\mathcal{V}$. We
  will show that $\mathcal{V}$ is in fact the hyperplane
  $S^2_0\mathbb{R}^n\wedge\id\oplus\mathcal{W}=\{\R|\tr(\R)=0\}$. Since
  $\Id\in\mathcal{C}$, this will imply that
  $\mathcal{C}=\{\R|\tr(\R)\geq 0\}$.

  We argue by contradiction and assume that
  $\mathcal{V}=\mathcal{W}$.

  \begin{claim}\label{claim_notTracelessInC}
    $\mathcal{C}\cap S^2_0\mathbb{R}^n\wedge\id=\{0\}$
  \end{claim}
  $\mathcal{C}\cap S^2_0\mathbb{R}^n\wedge\id$ is a convex cone in
  $S^2_0\mathbb{R}^n\wedge\id$ which is invariant under the action of
  $O(n,\mathbb{R})$. Since $O(n,\mathbb{R})$ acts irreducibly on
  $S^2_0\mathbb{R}^n\wedge\id$, if $\mathcal{C}$ contains a non-zero
  curvature operator in $S^2_0\mathbb{R}^n\wedge\id$ then, by
  Proposition \ref{prop_coneInIrredRep},
  $S^2_0\mathbb{R}^n\wedge\id\subset\mathcal{C}$ and
  $S^2_0\mathbb{R}^n\wedge\id\subset\mathcal{V}$, which contradicts
  the assumption that $\mathcal{V}=\mathcal{W}$. The claim is proved.

  As in Lemma \ref{lem_computeBR0W}, let $\R_0$ be the traceless Ricci
  part of the curvature operator of
  $\mathbb{S}^{n-2}\times\mathbb{H}^2$, where each factor
  is endowed with its constant curvature metric of curvature +1 or -1.

  Let $\tau$ be the greatest $t$ such that
  $\Id+t\R_0\in\mathcal{C}$. $\tau$ is positive and finite : it is
  positive because $\Id$ is in the interior of $\mathcal{C}$, it is
  finite because if $\Id+t\R_0$ stays in $\mathcal{C}$ as $t$ goe to
  infinity, then $\frac{1}{t}\Id+\R_0$ stays in $\mathcal{C}$ and
  $\R_0$ is in $\mathcal{C}$, which contradicts Claim
  \ref{claim_notTracelessInC}.
  We set $\R=\Id+\tau\R_0$. The maximality of $\tau$ implies that
  $\R\in\partial\mathcal{C}$.

    Let $\Wop=Q(\R_0)_\mathcal{W}\in\mathcal{W}$.
  Using Proposition \ref{prop_boudaryPreservedByAddingV}, we have that
  for any $t\in\mathbb{R}$ :
  \[\R+t\Wop\in\partial\mathcal{C},\]
  which
  implies :
  $Q(\R+t\Wop)=Q(\R)+2tB(\R,\Wop)+t^2Q(\Wop) \in T_{\R}\mathcal{C}$ since
  $\mathcal{C}$ is Ricci flow invariant.

  By Lemma \ref{lem_QandIrreds}, $-t^2 Q(\Wop)\in\mathcal{W}$. Since
  $\mathcal{W}\subset T_{\R}\mathcal{C}$ and $T_{\R}\mathcal{C}$ is a
  convex cone, this
  implies that $Q(\R)+2tB(\R,\Wop)\in T_{\R}\mathcal{C}$. Therefore, we
  have that $\frac{1}{2t}Q(\R)+B(\R,\Wop)\in
  T_{\R}\mathcal{C}$. Letting $t$ go to infinity (and using that
  $T_{\R}\mathcal{C}$ is closed), we then
  have that :
  \[B(\R,\Wop)\in T_{\R}\mathcal{C}.\]
  Moreover, using Lemmas \ref{lem_QandIrreds} and
  \ref{lem_computeBR0W} :
  \[B(\R,\Wop)=B(\Id,\Wop)+\tau B(\R_0,\Wop)=a\tau\R_0.\]

  We have proved that $\R_0\in T_{\R}\mathcal{C}$. This
  implies that there is some $t>0$ such that
  $\R+t\R_0\in\mathcal{C}$, that is to say
  $\Id+(\tau+t)\R_0\in\mathcal{C}$. This contradicts the definition of $\tau$.
  \end{proof}

\section{Ricci flow invariant cones containing the curvature operator of an Einstein symmetric space}
\label{sec:ricci-flow-invariant}
In this section, we prove Corollary
\ref{cor_coneWithEinsteinSymInInterior}. We first need a couple of lemmas.

\begin{lem}
  Let $(M^n,g_0)$ be a non negatively curved Einstein symmetric space,
  then its curvature operator $\R$ satisfies :
  \begin{equation}
  Q(\R)=\lambda\R\label{QRlambdaR}
  \end{equation}
  where $\lambda>0$ is such that $\ric_{g_0}={\lambda}g_0$.
\end{lem}
\begin{rem}
  Since the equation $Q(\R)=\lambda\R$ is invariant under the action
  $O(n,\mathbb{R})$, it can be seen either as an equation in
  $S^2_B\Lambda^2T_xM$ for some $x$ in $M$ with $\R$ the curvature
  operator of $g_0$ at $x$, or as an equation in
  $S^2_B\Lambda^2\mathbb{R}^n$, where $\R$ is the expression of the
  curvature operator of $g_0$ at some point $x$ with respect to some
  orthonormal basis of $T_xM$.
\end{rem}
\begin{proof}
  Since $(M^n,g_0)$ is Einstein, we can use Proposition 3 of
  \cite{brendle2010einstein} to get that (note that our definition of
  $Q$ differs from the one used by Brendle by a factor of 2) :
  \[\Delta_{g_0}\R_{g_0}+ 2Q(\R_{g_0})=\lambda\R.\]
  Then, since $g_0$ is symmetric, $\R_{g_0}$ is parallel and
  $\Delta_{g_0}\R_{g_0}=0$. This proves that $Q(\R_{g_0})=\lambda \R_{g_0}$.
\end{proof}
\begin{prop}
  If $\mathcal{C}$ is a Ricci flow invariant cone which contains the
  curvature operator $\R$ of an Einstein symmetric space in its
  interior, then the Weyl part $\R_{\mathcal{W}}$ of $\R$ is in $\mathcal{C}$.
\end{prop}
\begin{proof}
  By the previous lemma, $Q(\R)=\lambda\R$. Rescaling $\R$, we can
  assume that $Q(\R)=\R$. We decompose $\R$ along the decomposition
  \eqref{eq:decomp} : $\R=\R_{\Id}+\R_{\mathcal{W}}$. Since
  $B(\R_{\Id},\R_{\mathcal{W}})=0$, we have that
  $\R_{\Id}+\R_{\mathcal{W}}=\R=Q(\R)=Q(\R_{\Id})+Q(\R_{\mathcal{W}})$,
  which implies that $Q(\R_{\Id})=\R_{\Id}$ and $Q(\R_{\mathcal{W}})=\R_{\mathcal{W}}$.

  Since $\R$ is in the interior of $\mathcal{C}$, $\bar{\R}=\R-\eps \Id$ is also
  in $\mathcal{C}$ for some $\eps>0$ small enough. Then we have that
  $Q(\bar{\R}_{\Id})=(1-\eps)\bar{\R}_{\Id}$ and
  $Q(\bar{R}_{\mathcal{W}})=\bar{\R}_{\mathcal{W}}$. This allows us to
  explicitly write the solution to Hamilton's ODE
  $\frac{d}{dt}\bar{\R}(t)=Q(\bar{\R}(t))$ with initial condition
  $\bar{\R}(0)=\bar{\R}$, which is defined for $t<1$ :
  \[\bar{\R}(t)=\frac{1}{1-(1-\eps)
    t}\bar{\R}_{\Id}+\frac{1}{1-t}\bar{\R}_{\mathcal{W}}.\]

  Since $\mathcal{C}$ is Ricci flow invariant, $\bar{\R}(t)$ is in
  $\mathcal{C}$ for all $t\in [0,1)$, and since $\mathcal{C}$ is a
  cone :
  \[\forall t\in[0,1)\quad (1-t)\bar{\R}(t)=\frac{1-t}{1-(1-\eps)
    t}\bar{\R}_{\Id}+\bar{\R}_{\mathcal{W}}\in\mathcal{C}.\]
  Letting $t$ go to 1 and using that $\mathcal{C}$ is closed, we have
  that $\bar{\R}_{\mathcal{W}}={\R}_{\mathcal{W}}\in\mathcal{C}$.
\end{proof}
We can now prove Corollary \ref{cor_coneWithEinsteinSymInInterior} :
\begin{proof}[Proof (of the Corollary) :]
  Under the assumption of the corollary, the previous proposition
  shows that $\mathcal{C}$ contains the Weyl part $\R_\mathcal{W}$ of the curvature
  operator a non trivial symmetric space. In particular,
  $\R_\mathcal{W}\neq 0$. Using Proposition \ref{prop_coneInIrredRep},
  this implies that $\mathcal{W}\subset\mathcal{C}$ and we can
  apply Theorem \ref{thm:main}. This concludes the proof.
\end{proof}

\section{Curvature cones containing a conformally flat scalar flat curvature operator}
In this section we prove Theorems \ref{thm_ex_dim4} and \ref{thm_cone_with_cfsf}.
We begin with Theorem \ref{thm_ex_dim4}, whose statement we recall :
\begin{thm}
Let $\tilde{\mathcal{C}}\subset S^2_B\Lambda^2\mathbb{R}^4$ be any Ricci flow invariant cone which
contains all nonnegative curvature operators and  let :
\[\mathcal{C}=\left\{\R\left|\frac{\scal}{12}+W\in\tilde{\mathcal{C}}\right .\right\}\subset S^2_B\Lambda^2\mathbb{R}^4\]
where $\scal$ and $W$ are the scalar curvature and the Weyl curvature tensor of $\R$.

  Then $\mathcal{C}$ is a Ricci flow invariant cone which contains all conformally flat scalar flat curvature operators.
\end{thm}
\begin{rem}\label{sec:rem_wilk_cones}
It is interesting to ask whether these cones can be found by Wilking's
method \cite{2010arXiv1011.3561W}. If one sets $\tilde{\mathcal{C}}$
to be the cone of nonnegative curvature operators, $\mathcal{C}$ is a
Wilking cone built from the $SO(n,\mathbb{C})$ invarainat set
$S=\Lambda^2_+\mathbb{C}^4\cup\Lambda^2_-\mathbb{C}^4$, where
$\Lambda^2_+\mathbb{C}^4$ is the space of selfdual complex two forms
and $\Lambda^2_-\mathbb{C}^4$ is the space of antiselfdual complex two
forms. Similarly, setting $\tilde{\mathcal{C}}$ to be the PIC cone, we
have that $\mathcal{C}$ is also the PIC cone, which is also a Wilking
cone.

However, if $\tilde{\mathcal{C}}$ is the cone of 2-nonnegative
curvature operators, then $\mathcal{C}$ is not a Wilking cone. We
sketch the proof of this fact here, for relevant notations, see
\cite{2010arXiv1011.3561W}.

Assume that $\mathcal{C}=\{\R\ |\
\R_{\Id}+\R_\mathcal{W}\in\tilde{\mathcal{C}}\}$ is a Wilking cone,
that is
\[\mathcal{C}=\{\R\ |\ \forall\omega\in S,\
\ps{\R\omega}{\bar{\omega}}\geq 0\}\]
for some
$S\subset\Lambda^2\mathbb{C}^4\simeq\mathfrak{so}(4,\mathbb{C})$ which is invariant under the natural
action of $SO(4,\mathbb{C})$. One shows that
$\tilde{\mathcal{C}}\subset\mathcal{C}$, which implies that
$S\subset S_1=\{\omega\in\mathfrak{so}(4,\mathbb{C})\ |\
\tr{\omega^2}=0\}$. Moreover, one can show that a Wilking cone
contains $S^2_0\mathbb{R}^4\wedge\id$ if and only if
$S\subset S_2=\Lambda^2_+\mathbb{C}^4\cup\Lambda^2_-\mathbb{C}^4$. This
implies that \[\mathcal{C}\supset \{\R\ |\ \forall\omega\in S_1\cap S_2,\
\ps{\R\omega}{\bar{\omega}}\geq 0\}\]
which is the cone of operators
$\R$ such that the restrictions of $\R_{\Id}+\R_\mathcal{W}$ to
$\Lambda^2_+\mathbb{R}^4$ and $\Lambda^2_-\mathbb{R}^4$ are (separately)
2-nonnegative, which is a weaker condition than asking
$\R_{\Id}+\R_\mathcal{W}$ to be 2-nonnegative on the full $\Lambda^2\mathbb{R}^4$.
\end{rem}

We now prove the theorem.
\begin{proof}
  Let $\R=\R_{\Id}+\R_0+\R_\mathcal{W}\in\partial\mathcal{C}$. Note
  that using decomposition \eqref{eq:decomp}, we have :
  \[\frac{\scal}{12}+W=\R_{\Id}+\R_\mathcal{W}\in\tilde{\mathcal{C}}.\]
  To simplify notations, we define $\R_\mathcal{E}=\R_{\Id}+\R_\mathcal{W}$
  We want to show that :
  \[Q(\R)\in T_{\R}{\mathcal{C}}=\left\{\Lop|\Lop_\mathcal{E}\in T_{\R_\mathcal{E}}\tilde{\mathcal{C}} \right\}.\]
  Since
  $\R_\mathcal{E}\in\partial\mathcal{\tilde{C}}$, we
  have :
  \begin{equation}
  Q(\R_\mathcal{E})=Q(\R_{\Id}+\R_\mathcal{W})=Q(\R_{\Id})+Q(\R_\mathcal{W})\in T_{\mathcal{E}}\tilde{\mathcal{C}}.\label{eq:1}
  \end{equation}
  We have :
  \begin{align*}
    Q(\R)=\ &Q(\R_{\Id})+Q(\R_0)+Q(\R_\mathcal{W})\\
    &+2Q(\R_{\Id},\R_0)+2Q(\R_{\Id},\R_\mathcal{W})+2Q(\R_0,\R_\mathcal{W}).
  \end{align*}
  Thanks to Lemma \ref{lem_QandIrreds}, all the terms of the second line belong to
  $S^2_0\mathbb{R}^n\wedge\id\subset\mathcal{C}$, and $Q(\R_{\Id})\in\mathbb{R}\Id$, $Q(\R_\mathcal{W})\in\mathcal{W}$. So :
  \[Q(\R)_{\Id}=Q(\R_{\Id})+Q(\R_0)_{\Id}\]
  and
  \[Q(\R)_\mathcal{W}=Q(\R_0)_\mathcal{W}+Q(\R_\mathcal{W}).\]
  Thus :
  \begin{align*}
    Q(\R)_{\Id}+Q(\R)_\mathcal{W}=\ & Q(\R_{\Id})
  +Q(\R_\mathcal{W})\\
  &+ Q(\R_0)_{\Id}+Q(\R_0)_\mathcal{W}.
  \end{align*}
  Thanks to \eqref{eq:1}, the first line is
  in $T_{\R_\mathcal{E}}\tilde{\mathcal{C}}$. We will prove
  that the second line is in fact $\tilde{\mathcal{C}}\subset
  T_{\R_\mathcal{E}}\tilde{\mathcal{C}}$. This will show that
  $\mathcal{C}$ is Ricci flow invariant.

  We now write down the second line $\Lop=Q(\R_0)_{\Id}+Q(\R_0)_\mathcal{W}$
  explicitly in term of the traceless Ricci tensor $\ric_0$ of
  $\R$. For this, we use the formula in Lemma 2.2 in \cite{MR2415394},
  to which we subtract the traceless Ricci part, and then specialize
  to $n=4$  :
  \begin{align*}
    \Lop=Q(\R_0)_{\Id}+Q(\R_0)_\mathcal{W}=
    &\frac{1}{2}\ric_0\wedge\ric_0+\frac{1}{2}\ric_0^2\wedge\id.
  \end{align*}
  If $(e_i)$ is a basis of eigenvectors of $\ric_0$ with eigenvalues
  $\lambda_i$, then the $(e_i\wedge e_j)$ form a basis of eigenvectors
  of $\Lop$ with eigenvalues :
  \begin{align*}
  \mu_{ij}=
  \frac{\lambda_i\lambda_j}{2}+\frac{\lambda_i^2+\lambda_j^2}{4}=\frac{(\lambda_i+\lambda_j)^2}{2}\geq
  0.
  \end{align*}
  This shows that $\Lop$ is nonnegative, and thus is in $\tilde{\mathcal{C}}$.
\end{proof}
We now prove Theorem \ref{thm_cone_with_cfsf}.
\begin{proof}
  Let $\mathcal{C}$ be a Ricci flow invariant cone which contains
  $S^2_0\mathbb{R}^n\wedge\id$ and all nonnegative curvature
  operators, with $n\geq 5$.

  Consider $\mathbb{S}^{n-2}\times\mathbb{H}^2$ with its product metric
  where the first factor has constant curvature $1$ and the second factor has
  constant curvature $-1$. It is conformally flat and has positive
  scalar curvature. Therefore its curvature operator $\R$ lies in the
  interior of $\mathcal{C}$.

  Now since $\mathcal{C}$ contains all non-negative curvature
  operators, it contains the curvature operator $\bar{\R}$ of
  $\mathbb{R}^{n-2}\times\mathbb{S}^2$ with its product metric. Thus,
  for any $a>0$, $b>0$, $a\R+b\bar{\R}$ is in the interior of
  $\mathcal{C}$.

  Finally, $\R+(n-3)\bar{\R}$ is the curvature operator of the
  product Einstein metric on $\mathbb{S}^{n-2}\times\mathbb{S}^2$,
  which is a symmetric space. Corollary
  \ref{cor_coneWithEinsteinSymInInterior} then implies that $\mathcal{C}=\mathcal{C}_{\scal}$.
\end{proof}

\appendix
\section{Generalities about cones invariant under the action of a Lie group}
\label{sec:gener-about-cones}
We prove here some elementary facts about convex cones in a vector
space which are invariant under the action of a Lie group. The example
we have in mind is of course a curvature cone in
$S^2_B\Lambda^2\mathbb{R}^n$. These results are probably known in some circles but we describe them here for the sake of completeness.

For this section, we will use the following notations: $G$ is a compact Lie group, $E$ is a Euclidean vector space with inner product $\ps{\ }{\ }$ on which $G$ acts by linear isometries, and $\mathcal{C}\subset E $ is a closed convex cone which is invariant under the action of $G$. The action of an element $g\in G$ on $E$ will be denoted by $x\in E\mapsto g.x\in E$.

The tangent cone to $\mathcal{C}$ at a point $x\in\mathcal{C}$ is
defined as follows :
\[T_x\mathcal{C}=\{v\in E\ |\ \exists t>0\ x+tv\in\mathcal{C}\}.\]
This a closed convex cone in $E$, however, it is not $G$ invariant in general.

\begin{prop}\label{lem_maxVectorSpace}
  There exist a unique vector subspace $\mathcal{V}$ of $E$
  which is included in $\mathcal{C}$ such that:
  \begin{itemize}
  \item any vector subspace satisfying $\mathcal{V} '\subset\mathcal{C}$
    satisfies $\mathcal{V} '\subset \mathcal{V}$,
  \item $\mathcal{V}$ is $G$-invariant.
  \end{itemize}
\end{prop}
\begin{proof}
  Consider two vector spaces $\mathcal{V}_1,\mathcal{V}_2\subset\mathcal{C}$. Since
  $\mathcal{C}$ is a convex cone, $\mathcal{V}_1\oplus
  \mathcal{V}_2\subset\mathcal{C}$. This shows that there is a biggest subspace
  $\mathcal{V}\subset\mathcal{C}$.

  We now show that $\mathcal{V}$ is $G$ invariant. Let $g\in
  G$ and $v\in \mathcal{V}$. Let $L=\mathbb{R}v$, then
  $g.L\subset \mathcal{C}$. Thus, since $V$ is the biggest vector space
  in $\mathcal{C}$, $g.L\subset \mathcal{V}$, and $gv\in \mathcal{V}$.
\end{proof}
In particular, $\mathcal{V}$ is a subrepresentation of $E$. In
particular, if $E$ splits as the sum of irreducibles $\bigoplus_{i\in
  I} E_i$, then $\mathcal{V}=\bigoplus_{i\in
  J} E_i$ for some subset $J$ of $I$.

\begin{prop}\label{prop_VboundaryC}
  $\mathcal{V}\subset\partial\mathcal{C}$, except if $\mathcal{C}=E$.
\end{prop}
\begin{proof}
  Assume there is some $x$ which is in $\mathcal{V}$ and in the interior of
  $\mathcal{C}$. There is a neighborhood $O$ of $x$ which is included
  in $\mathcal{C}$. Since $-x$ is in $\mathcal{V}$ and $\mathcal{C}$ is a
  convex cone, $O-x$ is a neighborhood of $0$ contained in
  $\mathcal{C}$. Since $\mathcal{C}$ is a convex cone, we then have that $\mathcal{C}=E$.
\end{proof}
\begin{prop}
  If $x\in\partial\mathcal{C}$, then $\mathcal{V}\subset T_x\mathcal{C}$.
\end{prop}
\begin{proof}
  This just follows from the general fact that $\mathcal{C}\subset T_x\mathcal{C}$.
\end{proof}
\begin{prop}\label{prop_boudaryPreservedByAddingV}
  If $x\in\partial\mathcal{C}$ and $v\in\mathcal{V}$, then $x+v\in\partial\mathcal{C}$.
\end{prop}
\begin{proof}
  If $x+v$ is in the interior of $\mathcal{C}$, one can find an
  open set $O\subset\mathcal{C}$ containing $x+v$, and $O-v$ is
  a neighborhood of $x$ contained in $\mathcal{C}$, a contradiction.
\end{proof}

\begin{prop}\label{prop_coneInIrredRep}
Assume that $E$ is a direct sum of non trivial
representation of $G$, then $\mathcal{C}$ is a subrepresentation of $E$.
\end{prop}
\begin{proof}
  We only need to show that $\mathcal{C}$ is a vector subspace of $E$. Since
  $\mathcal{C}$ is a convex cone, it is stable under linear
  combination with nonnegative coefficients. Therefore we only need to
  show that $\mathcal{C}$ is stable under $x\mapsto -x$.

  We argue by contradiction. Assume that $\mathcal{C}$ is not stable
  under $x\mapsto -x$. Consider the dual cone
  $\mathcal{C}^*=\{v\in E|\forall x\in\mathcal{C},\ps{v}{x}\geq
  0\}$. Then $\mathcal{C}^*$ is also a convex $G$-invariant cone, it
  is easy to see that $\mathcal{C}^*$ is also not a vector subspace of
  $E$. This implies that there is some $v$ in $\mathcal{C}^*$ such
  that $-v$ does not belong to $\mathcal{C}^*$. In particular, there is
  some $x_0\in\mathcal{C}$ such that $\ps{v}{x_0}>0$.

  Consider now:
\[\tilde{v}=\int_G g.v dg\]
where $dg$ is a Haar measure on $G$. Then $\tilde{v}$ is not zero because:
\[\ps{\tilde{v}}{x_0}=\int_G\ps{g.v}{x_0}dg>0\]
since $g\mapsto\ps{g.v}{x_0}=\ps{v}{g^{-1}x_0}$ is a continuous nonnegative function which is strictly positive at the neutral element of $G$. Moreover, for any $g'\in G$:
\[g'.\tilde{v}=\int_G g.(g'.v)dg=\int_G(gg').vdg=\int_G g.v dg=\tilde{v}.\]
This shows that $\mathbb{R}\tilde{v}\subset E$ is an irreducible
subrepresentation of $E$ which is trivial, a contradiction.
\end{proof}

\bibliographystyle{alpha}
\bibliography{noncol3RF}
\end{document}